\documentclass[abstracton,numbers=noendperiod]{scrartcl}

\usepackage[T1]{fontenc}
\usepackage[utf8]{inputenc}

\usepackage{amsmath}
\usepackage{amssymb}
\usepackage{amsthm}
\usepackage[backend=bibtex,hyperref=true,backref=true,style=alphabetic,maxnames=99]{biblatex}
\usepackage{booktabs}
\usepackage{csquotes}
\usepackage{enumitem}
\usepackage{mathrsfs}
\usepackage{mathtools}
\usepackage{microtype}
\usepackage{nth}

\DeclareFieldFormat[online]{title}{\mkbibquote{#1\isdot}}
\DeclareFieldFormat[online]{date}{(#1)}
\renewbibmacro{in:}{}
\DeclareFieldFormat[article]{pages}{#1\isdot}

\usepackage{bm}
\usepackage{hyperref}
\usepackage{tgheros}
\usepackage{tgpagella}
\usepackage{eulervm}

\usepackage{newunicodechar}
\newunicodechar{ō}{\=o}

\usepackage{tikz}
\usetikzlibrary{arrows,matrix,positioning,calc}
\tikzstyle{dmatrix}=[matrix of math nodes,row sep=2.5em, column sep=2.5em,
text height=1.5ex, text depth=0.25ex] 

\newcommand{\GrassmannianSymbol}{G}
\newcommand{\NormalBundleSymbol}{\mathcal{N}}
\newcommand{\TautBundleSymbol}{S}
\newcommand{\UnivQuotBundleSymbol}{Q}

\newcommand{\NormalBundle}{\NormalBundleSymbol}
\newcommand{\MyNormalBundle}{\NormalBundle_{C/\PP^7}}
\newcommand{\ExplicitNormalBundle}{\bigwedge\nolimits^2 S_C^\vee}
\newcommand{\MyGrassmannian}{\GrassmannianSymbol}
\newcommand{\Grassmannian}[2]{\GrassmannianSymbol(#1,#2)}
\newcommand{\TautBundle}{\TautBundleSymbol}
\newcommand{\UnivQuotBundle}{\UnivQuotBundleSymbol}

\DeclareMathOperator{\Cliff}{Cliff}

\DeclareMathOperator{\Hom}{Hom}

\DeclareMathOperator{\SpG}{SpG}
\DeclareMathOperator{\Sym}{Sym}

\DeclareMathOperator{\coker}{coker}

\DeclareMathOperator{\image}{im}

\DeclareMathOperator{\rk}{rk}

\newcommand{\Z} {\mathbb{Z}}

\newcommand{\PP}{\mathbb{P}}
\renewcommand{\P} {\mathbb{P}}

\newtheorem{proposition}{Proposition}[section]
\newtheorem{theorem}[proposition]{Theorem}
\newtheorem*{theorem*}{Theorem}
\newtheorem{corollary}[proposition]{Corollary}
\newtheorem*{corollary*}{Corollary}
\newtheorem{conjecture}[proposition]{Conjecture}
\newtheorem*{conjecture*}{Conjecture}
\newtheorem{lemma}[proposition]{Lemma}
\newtheorem*{lemma*}{Lemma}

\theoremstyle{definition}
\newtheorem{definition}[proposition]{Definition}
\newtheorem*{definition*}{Definition}
\newtheorem{example}[proposition]{Example}
\newtheorem*{example*}{Example}
\newtheorem{remark}[proposition]{Remark}
\newtheorem*{remark*}{Remark}

\newtheorem*{question*}{Question}

\newtheorem*{result*}{Result}

\newcommand{\MgSymbol}{\mathcal{M}}

\newcommand{\MgBNLocus}[3]{\MgSymbol_{#1,#3}^{#2}}

\newcommand{\Addresses}{{
  \bigskip
  \footnotesize
  \textsc{Institut für Algebraische Geometrie, Leibniz Universität Hannover, Welfengarten 1}\par
  30167 \textsc{Hannover, Germany}\par\nopagebreak\vspace{3mm}
  \textsc{Riemann Center for Geometry and Physics, Leibniz Universität Hannover, Appelstrasse 2}\par
  30167 \textsc{Hannover, Germany}\par\nopagebreak
  \textit{E-mail address}:  \texttt{math@gregorbruns.eu}

}}

\newcommand{\wedgesc}{\bigwedge^2 S_C^\vee}

\numberwithin{equation}{section}

\title{The normal bundle of canonical genus $8$ curves}
\author{Gregor Bruns}
\date{}

\begin{document}

\maketitle

\begin{abstract}
  We study the stability of the normal bundle of canonical
  genus $8$ curves and prove that on a general curve the bundle is stable.
  The proof rests on Mukai's description of these curves as
  linear sections of a Grassmannian $\Grassmannian{2}{6}$.
  This is the next case of a conjecture by M.~Aprodu, G.~Farkas,
  and A.~Ortega:  the general canonical curve of every genus $g \geq 7$
  should have stable normal bundle.
  We also give some more evidence for this conjecture in higher genus.
\end{abstract}

\section{Introduction}

To an embedded curve $C \hookrightarrow \PP^r$ with $L = \mathcal{O}_C(1)$
one can naturally attach two distinguished vector bundles.  First there is
the restricted cotangent bundle $\Omega_{\PP^r}(1)|_C$ which by the Euler sequence
coincides with the kernel bundle $M_L$ in the evaluation exact sequence
\begin{equation*}
  0 \rightarrow M_L \rightarrow H^0(C, L) \otimes \mathcal{O}_C \rightarrow L \rightarrow 0.
\end{equation*}
A lot is known about the geometry of $M_L$.  For instance, it governs the syzygies of
the embedding of $C$, and a simple argument (like in \autocite{EusenSchreyer2012}) shows
that in most cases $M_L$ is a stable vector bundle.
In particular this is true for the canonical embedding, i.e., for $L = \omega_C$.

The second distinguished vector bundle is the normal bundle $\mathcal{N}_{C/\PP^r}$
of the embedding.  Of particular interest is again the normal bundle
$\mathcal{N}_{C/\PP^{g-1}}$ of the canonical embedding of $C$.
As it turns out,
the study of normal bundles is much more delicate than the study of
the restricted cotangent bundles.  In particular, not much is known about stability
properties, not even for canonical curves. 

Of course a non-hyperelliptic canonical curve of genus $3$ is a plane quartic,
hence the normal bundle is a line bundle and therefore stable.
Moving on, canonical curves of genera $4$ and $5$ are complete intersections,
hence they do not have stable normal bundles.
In \autocite{AFO2016} it was shown that a
tetragonal curve of genus at least $6$ has unstable normal bundle.
This applies in particular to all canonical curves of genus $6$.

On the other hand, in the same paper \autocite{AFO2016}
the authors prove the first nontrivial positive result:
in genus $7$ the general canonical curve has stable normal bundle.
Observe that, starting from genus $7$, the general curve is no longer tetragonal.
Based on these considerations, the authors then made the following conjecture:
\begin{conjecture}[{\autocite[Conjecture 0.4]{AFO2016}}]
  \label{conj:afo}
  The normal bundle of a general canonical curve of genus $g \geq 7$ is stable.
\end{conjecture}
In this paper we consider the next case, genus $8$, and provide a full solution:
\begin{theorem}
  \label{thm:main-theorem}
  The normal bundle of a canonical curve $C$ of genus $8$ is stable
  if and only if the curve does not have a $\mathfrak{g}^2_7$.
  Furthermore, it is polystable if and only if the curve is not tetragonal.
\end{theorem}
Note that a tetragonal curve always has a $\mathfrak{g}^2_7$ \autocite[Lemma 3.8]{Mukai1993}.

We sketch the proof of Theorem \ref{thm:main-theorem}.
The first ingredient is a construction of S.\@ Mukai, found in \autocite{Mukai1993},
showing that the general genus $8$ curve is a transversal linear section of a
Grassmannian $G = \Grassmannian{2}{6}$ in its Plücker embedding.
The normal bundle of $C$ then arises as the pullback of the normal bundle of $G$,
which has a description in terms of the tautological bundle $\TautBundle$ on $G$.
More precisely, we have $\NormalBundle_{C/\PP^7}(-1) = \ExplicitNormalBundle$,
where $\TautBundle_C$ is the restriction of $\TautBundle$ to $C$.

In Section \ref{subsec:stability-tautological-bundle} we then prove
that $\TautBundle_C$ is stable.  Here we again use essential properties
of the Mukai embedding $C \subseteq G$ to describe the geometry of $S_C$.
Since exterior powers of stable bundles are polystable, we are left with proving
that $\NormalBundle_{C/\PP^7}$ does not decompose.  We achieve this in
Section \ref{subsec:stability-normal-bundle-g8} through a case-by-case analysis.

If $W^2_7(C) \not= \emptyset$ but the curve is not tetragonal,
we show in Section \ref{sec:semistability-g27} that the normal bundle is not stable,
but still polystable.  The essential ingredient is the existence of complete
intersection models of these curves in $\PP^2\times \PP^2$,
constructed by M.~Ide and S.~Mukai in \autocite{MI2003}.  Interestingly,
for a suitable limit definition of $S_C$, the description of $\NormalBundle_{C/\PP^7}(-1)$
as $\ExplicitNormalBundle$ still holds.

Finally, in two appendices, we first reprove the result that no canonical genus $6$ curve
has stable normal bundle using Mukai's Grassmannian construction.
Secondly, we give some remarks about the case which comes next,
i.e., genus $9$ canonical curves, which seems to be considerably more difficult
and might require new techniques.
\begin{remark}
  For more general linear systems than the canonical one,
  not much is known either.
  On the one hand, embeddings of low degree,
  as well as certain embeddings by non-complete linear series, are expected to be unstable.
  Normal bundles of rational and elliptic space curves have been studied
  and shown to be unstable in many cases, for instance in \autocite{EvdV1982, HS1983}.
  
  On the other hand, at least conjecturally, embeddings by complete linear systems
  of high degree are expected to be stable. 
  In \autocite{EL1992}, Ein and Lazarsfeld prove that normal bundles of elliptic curves
  in an embedding of degree at least $3$ are polystable and they show that for genus $g \geq 2$
  and degree $d \gg 0$ the normal bundles are at least not too unstable.
\end{remark}
\begin{remark}
  The question about stability properties of normal bundles can also
  be posed for polarized K3 surfaces.
  There is an intimate relation between the geometry of K3s and the geometry of canonical curves.
  In particular, Mukai's constructions originally were set up for Fano $3$-folds
  and K3 surfaces, with curves as an application (see \autocite{Mukai1988}).
  Using this, we can almost immediately lift results about normal bundles up to genus $9$ from
  a canonical curve $C$ to a polarized K3 surface $(S, \mathcal{O}_S(C))$.
  It may be worth investigating the opposite direction:  consider stability
  of normal bundles of polarized K3 surfaces and try to infer something
  about the normal bundle of their hyperplane sections.
\end{remark}



\paragraph{Acknowledgements}

The content of this article was part of my Ph.D.\@ thesis.
First I would like to thank my Ph.D.\@ advisor Gavril Farkas
for suggesting this very interesting problem to me.
Thanks also goes to Carel Faber, my co-advisor, who hosted
me at Universiteit Utrecht during the time this paper
was conceived.  I greatly enjoyed the atmosphere and guidance there.
The research stay was generously supported by the IRTG 1800
of the Deutsche Forschungsgemeinschaft.

\section{Stability of the normal bundle for general curves}

\subsection{Facts about normal bundles of canonical curves}

Consider a non-hyperelliptic smooth curve of genus $g \geq 3$,
embedded by the canonical embedding in $\PP^{g-1}$.
We denote the normal bundle of this embedding by $\NormalBundle_{C/\PP^{g-1}}$.
Using the Euler exact sequence and the normal bundle exact sequence
we calculate that in this case
\begin{equation*}
  \det( \NormalBundle_{C/\PP^{g-1}} ) = \omega_C^{\otimes (g+1)},\quad
  \deg( \NormalBundle_{C/\PP^{g-1}} ) = (g+1)(2g - 2)
\end{equation*}
and we remark that $\NormalBundle_{C/\PP^{g-1}}$ is of rank $g - 2$.

To discuss stability properties of $\NormalBundle_{C/\PP^{g-1}}$
in a more precise way, recall the following definition:
\begin{definition}
  A vector bundle on a curve $C$ is called \emph{polystable}
  if it splits into the direct sum of stable bundles,
  all of the same slope.  
\end{definition}
Of course every stable bundle is polystable and polystable bundles are semistable.

We now consider, as an easy example,
the stability of the normal bundle of canonical curves in low genus,
briefly mentioned already in the introduction.

Every non-hyperelliptic genus $3$
curve is canonically embedded as a quartic in $\PP^2$, hence the normal
bundle is the line bundle $\mathcal{O}_C(4)$ and therefore stable.
For genus $4$ and $5$,
the general canonical curve is the complete intersection of
a cubic and a quadric, or three quadrics, respectively.  Since the
normal bundle of a complete intersection splits, we get
\begin{equation*}
  \NormalBundle_{C/\PP^{g-1}} = 
  \begin{cases}
    \mathcal{O}_C(2) \oplus \mathcal{O}_C(3), & g = 4\\
    \mathcal{O}_C(2) \oplus \mathcal{O}_C(2) \oplus \mathcal{O}_C(2), & g = 5.
  \end{cases}
\end{equation*}
We see that for genus $4$ the normal bundle is unstable
while it is polystable (but not stable) for genus $5$.

\subsection{Kernel bundles}
\label{subsec:kernel-bundles}

Now we introduce the main technical tool we need.
If $L$ is a line bundle on $C$ then we denote by $M_L$
the kernel of the evaluation morphism
\begin{equation}
  \label{eq:kernel-bundle-definition}
  H^0(C, L) \otimes \mathcal{O}_C \rightarrow L.
\end{equation}
As a subsheaf of the (free) vector bundle $H^0(C, L) \otimes \mathcal{O}_C$,
the sheaf $M_L$ is itself a vector bundle
and is called the \emph{kernel bundle} of $L$.
We denote its dual
by $Q_L = M_L^\vee$.  If $L$ is globally generated, then the evaluation map 
\eqref{eq:kernel-bundle-definition} is surjective and the sequence
\begin{equation*}
  0 \rightarrow M_L \rightarrow H^0(C, L) \otimes \mathcal{O}_C \rightarrow L \rightarrow 0
\end{equation*}
is also exact on the right.  We obtain the dual short exact sequence
\begin{equation*}
  0 \rightarrow L^{-1} \rightarrow H^0(C, L)^\vee \otimes \mathcal{O}_C \rightarrow Q_L \rightarrow 0
\end{equation*}
which shows that $Q_L$ is globally generated.  Taking global sections in this exact
sequence also shows $h^0(C, Q_L) \geq h^0(C, L)$ and we expect equality to hold.
We furthermore have that $\rk(Q_L) = h^0(C, L) - 1$ and $\det Q_L = L$.

The bundles $Q_L$ are often stable and in many cases we have a natural
isomorphism $H^0(C, L)^\vee \cong H^0(C, Q_L)$.
In this way, kernel bundles provide many examples of stable bundles
with unusually many global sections.  
They have been used to construct vector bundles with special
properties on curves and surfaces, e.g.\@ Ulrich bundles (\autocite{ES2003}),
and bundles violating Mercat's conjecture (\autocite{FO2012}).

Kernel bundles have also been used extensively in the context of Koszul cohomology,
for instance in construction of Koszul divisors on moduli
spaces of curves (see \autocite{Farkas2009}).
R.\@ Lazarsfeld used them to show that a general curve on a general K3 surface
satisfies the Brill--Noether--Petri theorem \autocite{Lazarsfeld1986}.
C.\@ Voisin made use of Kernel bundles in \autocite{Voisin2002, Voisin2005} to
prove Green's conjecture for generic curves of even and odd genus, respectively.

\subsection{Description of the normal bundle}
\label{subsec:normalbundle-g8-description}

Throughout the rest of this section we will assume that $C$ is a curve of genus $g = 8$,
canonically embedded in $\PP^7$, such that $W^2_7(C) = \emptyset$.
In other words, $C$ should be Brill--Noether general.
By the work of Mukai (\autocite{Mukai1993}), such a curve $C$ is a transversal
linear section of a Grassmannian $\MyGrassmannian = \Grassmannian{2}{6}$, embedded
by the Plücker embedding in $\PP^{14}$.  More precisely, there is a
$7$-plane $\PP^7 \subset \PP^{14}$ such that $C = \MyGrassmannian \cap \PP^7$.

The inclusion $C \hookrightarrow \MyGrassmannian$ is induced by the global
sections of the \emph{Mukai bundle} $E_C$ on $C$, an up to isomorphism
uniquely defined stable rank $2$ bundle with $h^0(C, E_C) = 6$.
If $\zeta$ is any $\mathfrak{g}^1_5$ on $C$ and $\eta = \omega_C \otimes \zeta^{-1} \in W^3_9(C)$
its Serre dual, then $E_C$ sits in the exact sequence
\begin{equation*}
  0 \rightarrow \zeta \rightarrow E_C \rightarrow \eta \rightarrow 0
\end{equation*}
which is split on global sections.

The normal bundle $\NormalBundle_{C/\PP^7}$
has rank $6$ and determinant $\omega_C^{\otimes 9}$.
We can use the Grassmannian embedding of $C$ to derive more
information about the structure of $\NormalBundle_{C/\PP^7}$.
Because $C = \PP^7 \cap \MyGrassmannian$,
we have the split exact sequence
\begin{equation*}
  0 \rightarrow \mathcal{N}_{C/\MyGrassmannian} \rightarrow \mathcal{N}_{C/\PP^{14}}
  \rightarrow \mathcal{N}_{C/\PP^7} \rightarrow 0 
\end{equation*}
and the exact sequence of normal bundles induced by the inclusions
$C \hookrightarrow \MyGrassmannian \hookrightarrow \PP^{14}$
\begin{equation*}
  0 \rightarrow \mathcal{N}_{C/\MyGrassmannian} \rightarrow \mathcal{N}_{C/\PP^{14}}
  \rightarrow \mathcal{N}_{\MyGrassmannian/\PP^{14}}|_C \rightarrow 0
\end{equation*}
which together imply $\NormalBundle_{C/\PP^7} = \NormalBundle_{\MyGrassmannian/\PP^{14}}|_C$.
But the normal bundle of a $\Grassmannian{2}{n}$ in its Plücker embedding
is explicitly given by
\begin{equation*}
  \NormalBundle_{\Grassmannian{2}{n}/\PP^N}
  = \bigwedge\nolimits^2 S^\vee \otimes \bigwedge\nolimits^2 Q
  = \bigwedge\nolimits^2 S^\vee \otimes \mathcal{O}_{\Grassmannian{2}{n}}(1).
\end{equation*}
A proof of this classical fact can be found, e.g., in \autocite{Manivel1998}.
Here $\TautBundle$ is the tautological bundle over $\Grassmannian{2}{n}$
and $\UnivQuotBundle$ the universal quotient bundle of rank $2$,
sitting in the exact sequence
$0 \rightarrow S \rightarrow \mathcal{O}_{\Grassmannian{2}{n}}^{\oplus n} \rightarrow Q \rightarrow 0$.
Hence if $\TautBundle_C$ denotes the restriction of $\TautBundle$ to $C$ then
\begin{equation*}
  \NormalBundle_{C/\PP^7}(-1) = \ExplicitNormalBundle.
\end{equation*}
In what follows we will prove stability of
this twist $\NormalBundle_{C/\PP^7}(-1)$.  Note that
$\det \NormalBundle_{C/\PP^7}(-1) = \omega_C^{\otimes 3}$
and the slope of this bundle is $\mu(\NormalBundle_{C/\PP^7}(-1)) = 7$.
\begin{remark}
  Several things are special about the case of genus $8$.
  In fact, this is the only genus $g \geq 6$ where the slope of the normal bundle is an integer.
  This happens if and only if
  \begin{equation*}
    (g-2) \mid 2(g-1)(g+1)
  \end{equation*}
  and since $(g-2)$ and $(g-1)$ have no common factors, we must have
  $(g-2) \mid 2(g+1)$ which implies that all divisors of $(g-2)$ are $1$, $2$ or $3$.
  So the only possibilities are $g = 3,4,5,8$.

  By the above explicit description of $\NormalBundle_{C/\PP^7}$ we also see that
  the normal bundle of a general canonical genus $8$ curve is self-dual up to twist,
  since for any vector bundle $\mathcal{F}$ of rank $4$ we have the duality
  \begin{equation*}
    \bigwedge\nolimits^2 \mathcal{F}
    = \Big( \bigwedge\nolimits^2 \mathcal{F} \Big)^\vee \otimes \det \mathcal{F}
  \end{equation*}
  and hence
  \begin{equation*}
    \MyNormalBundle^\vee(2) \cong \MyNormalBundle(-1).
  \end{equation*}
  By comparing the degrees of $\NormalBundle_{C/\PP^{g-1}}(k)$ and $\NormalBundle_{C/\PP^{g-1}}^\vee(l)$
  for $k,l\in \Z$ and arbitrary genus, we see that $g = 8$ is the only case
  where this could happen.
\end{remark}
Our strategy is then to derive
the stability of $\NormalBundle_{C/\PP^7}(-1)$ from the stability of $\TautBundle_C$.
Hence we first have to understand the bundle $\TautBundle_C$ better.
First of all noting some basic numerical facts, the slope is $\mu(S_C^\vee) = 7/2$.
We can restrict the universal exact
sequence over the Grassmannian $\Grassmannian{2}{6}$
to $C$ and obtain
\begin{equation*}
  0 \rightarrow S_C \rightarrow \mathcal{O}_C^{\oplus 6} \rightarrow E_C \rightarrow 0
\end{equation*}
where $E_C$ denotes, as before, the Mukai bundle of $C$.
This implies that $S_C^\vee$ is globally generated and $h^0(C, S_C^\vee) \geq 6$.
We also have $\det(S_C^\vee) = \omega_C$ and $\chi(S_C^\vee) = -14$,
hence $h^0(S_C^\vee) = h^1(S_C^\vee) - 14$.

\subsection{Stability of the tautological bundle}
\label{subsec:stability-tautological-bundle}

We are going to prove that the restriction of the tautological bundle $\TautBundle$ of $G(2,6)$
to $C$ is stable.
The following lemma will be instrumental in what follows.
\begin{lemma}
  \label{lem:morph-to-g15}
  Let $\zeta \in W^1_5(C)$.  Then there is a unique surjection $S_C^\vee \rightarrow \zeta$.
  This gives rise to an exact sequence
  \begin{equation*}
    0 \rightarrow Q_\eta \rightarrow S_C^\vee \rightarrow \zeta \rightarrow 0
  \end{equation*}
  where $Q_\eta$ is the dual of the kernel bundle of $\eta = \omega_C \otimes \zeta^{-1}$.
\end{lemma}
\begin{proof}
  The space $H^0(C, E_C \otimes \zeta^{-1}) = H^1(C, E_C \otimes \zeta)^\vee$
  is easily seen to be $1$-dimensional (for instance by Lemma 3.10 of \autocite{Mukai1993}),
  hence $h^0(C,E_C \otimes \zeta) = 11$.  Let $V = H^0(C, E_C)$.  Then by tensoring
  the exact sequence
  \begin{equation*}
    0 \rightarrow S_C \rightarrow V \otimes \mathcal{O}_C
    \rightarrow E_C \rightarrow 0
  \end{equation*}
  with $\zeta$ and taking cohomology we obtain
  \begin{equation*}
    0 \rightarrow H^0(C, S_C \otimes \zeta) \rightarrow V \otimes H^0(C, \zeta)
    \rightarrow H^0(C, E_C \otimes \zeta).
  \end{equation*}
  We want to show that the kernel $H^0(C, S_C \otimes \zeta) = \Hom(S_C^\vee, \zeta)$
  of the multiplication map
  $\mu\colon V \otimes H^0(C, \zeta) \rightarrow H^0(C, E_C \otimes \zeta)$
  is one-dimensional.
  Using the base point free pencil trick we write
  \begin{equation*}
    0 \rightarrow \zeta^{-1} \rightarrow H^0(C, \zeta) \otimes \mathcal{O}_C \rightarrow \zeta
    \rightarrow 0
  \end{equation*}
  and tensor this by $E_C$.  Taking cohomology, we see that the kernel of $\mu$
  is exactly $H^0(C, E_C \otimes \zeta^{-1})$, i.e., $\dim(\ker(\mu)) = 1$.
  
  This means there is a unique nonzero morphism
  $S_C^\vee \rightarrow \zeta$ which surjects onto a line bundle
  of degree $d \leq \deg(\zeta) = 5$.  This line bundle,
  as a quotient of $S_C^\vee$, must be globally generated and since $W^1_4(C) = \emptyset$
  it must be $\zeta$ itself.  We get an exact sequence
  \begin{equation*}
    0 \rightarrow F \rightarrow S_C^\vee \rightarrow \zeta \rightarrow 0
  \end{equation*}
  where $F$ is a rank $3$ vector bundle with $\det(F) = \eta$, the Serre dual of $\zeta$.
  Because $H^0(C, E_C) = H^0(C, \zeta) \oplus H^0(C, \eta)$ we get
  the following commutative diagram with exact columns and the upper two rows also exact:
  \begin{center}
    \begin{tikzpicture}
      \matrix[dmatrix] (m)
      {
        & 0 & 0 & 0 & \\
        0 & \zeta^{-1} & H^0(C,\zeta)\otimes\mathcal{O}_C & \zeta & 0\\
        0 & S_C & H^0(C, E) \otimes \mathcal{O}_C & E_C & 0\\
        0 & F^\vee & H^0(C,\eta) \otimes \mathcal{O}_C & \eta & 0\\
        & 0 & 0 & 0 & \\
      };
      \draw[->] (m-2-1) to (m-2-2);
      \draw[->] (m-2-2) to (m-2-3);
      \draw[->] (m-2-3) to (m-2-4);
      \draw[->] (m-2-4) to (m-2-5);
      \draw[->] (m-3-1) to (m-3-2);
      \draw[->] (m-3-2) to (m-3-3);
      \draw[->] (m-3-3) to (m-3-4);
      \draw[->] (m-3-4) to (m-3-5);
      \draw[->] (m-4-1) to (m-4-2);
      \draw[->] (m-4-2) to (m-4-3);
      \draw[->] (m-4-3) to (m-4-4);
      \draw[->] (m-4-4) to (m-4-5);

      \draw[->] (m-1-2) to (m-2-2);
      \draw[->] (m-2-2) to (m-3-2);
      \draw[->] (m-3-2) to (m-4-2);
      \draw[->] (m-4-2) to (m-5-2);
      \draw[->] (m-1-3) to (m-2-3);
      \draw[->] (m-2-3) to (m-3-3);
      \draw[->] (m-3-3) to (m-4-3);
      \draw[->] (m-4-3) to (m-5-3);
      \draw[->] (m-1-4) to (m-2-4);
      \draw[->] (m-2-4) to (m-3-4);
      \draw[->] (m-3-4) to (m-4-4);
      \draw[->] (m-4-4) to (m-5-4);
    \end{tikzpicture}
  \end{center}
  The Snake Lemma implies that the last row is exact as well,
  i.e., $F^\vee$ is actually $M_\eta$.
\end{proof}
\begin{remark}
  This shows in particular that $h^0(C, S_C^\vee) = 6$, since from the exact sequence
  $0 \rightarrow Q_\eta \rightarrow S_C^\vee \rightarrow \zeta \rightarrow 0$
  we get the bound
  \begin{equation*}
    h^0(C, S_C^\vee) \leq h^0(C, \zeta) + h^0(C, Q_\eta) = 2 + 4
  \end{equation*}
  and we already knew $h^0(C, S_C^\vee) \geq 6$ from before.
\end{remark}
The bundle $Q_\eta$ that appears in Lemma \ref{lem:morph-to-g15} plays an important
role in understanding the stability of $S_C^\vee$.  It is not too hard to show
that $Q_\eta$ is itself stable.
\begin{lemma}
  \label{lem:stability-kernel-g39}
  Let $\eta \in W^3_9(C)$ 
  and let $Q_{\eta}$ be the dual of the kernel bundle of $\eta$.
  Then $Q_{\eta}$ is stable and $H^0(C, Q_{\eta}) = H^0(C, \eta)^\vee$.
\end{lemma}
\begin{proof}
  The stability is proved in \autocite{EusenSchreyer2012}.

  Let $D$ be the intersection of $C$ and a quadrisecant line in
  $\PP^3$ where $C$ is embedded by $|\eta|$.
  Then $\eta(-D) = \zeta \in W^1_5(C)$.  We have a commutative diagram
  \begin{center}
    \begin{tikzpicture}
      \matrix[dmatrix] (m)
      {
        0 & \eta^{-1} & H^0(\eta)^\vee \otimes \mathcal{O}_C & Q_{\eta} & 0\\
        0 & \zeta^{-1} & H^0(\zeta)^\vee \otimes \mathcal{O}_C & \zeta & 0\\
      };
      \draw[->] (m-1-1) -- (m-1-2);
      \draw[->] (m-1-2) -- (m-1-3);
      \draw[->] (m-1-3) -- (m-1-4);
      \draw[->] (m-1-4) -- (m-1-5);
      \draw[->] (m-2-1) -- (m-2-2);
      \draw[->] (m-2-2) -- (m-2-3);
      \draw[->] (m-2-3) -- (m-2-4);
      \draw[->] (m-2-4) -- (m-2-5);
      \draw[->] (m-1-3) -- (m-2-3);
      \draw[->] (m-1-4) -- (m-2-4);
      \draw[right hook->] (m-1-2) -- (m-2-2);
    \end{tikzpicture}
  \end{center}
  and an exact sequence
  \begin{equation*}
    0 \rightarrow G \rightarrow Q_{\eta} \rightarrow \zeta \rightarrow 0
  \end{equation*}
  where $G$ is a rank $2$ bundle with $\det(G) = \mathcal{O}_C(D)$.
  Since $h^0(C, Q_{\eta}) \geq 4$ we have $h^0(C, G) \geq 2$.

  Now $G$ is the extension
  \begin{equation*}
    0 \rightarrow A \rightarrow G \rightarrow B \rightarrow 0
  \end{equation*}
  of two line bundles $A,B$ where we can assume $\deg(A) \geq 0$.
  By stability of $Q_{\eta}$ we have $\deg(A) \leq 2$ and hence $2 \leq \deg(B) \leq 4$.
  In any case, $h^0(C, A) = h^0(C, B) = 1$ and hence $h^0(C, G) = 2$.
  This implies $h^0(C, Q_{\eta}) = 4$.
\end{proof}
The following lemma is well-known.  
\begin{lemma}
  \label{lem:determinant-sections}
  Let $G$ be a globally generated vector bundle on $C$ with $H^0(C, G^\vee) = 0$.
  Then $h^0(C, \det(G)) \geq \rk(G) + 1$ with equality if and only
  if $G$ is the dual of the kernel bundle of $\det(G)$.
\end{lemma}
\begin{proof}
  Since $G$ can be generated by $\rk(G) + 1$ sections, we have a surjection
  $\mathcal{O}_C^{\rk(G) + 1} \rightarrow G \rightarrow 0$ whose kernel
  is $\det(G)^{-1}$.  Dualizing this sequence, we get
  \begin{equation*}
    0 \rightarrow G^\vee \rightarrow \mathcal{O}_C^{\rk(G) + 1} \rightarrow \det(G) \rightarrow 0
  \end{equation*}
  and using $H^0(C, G^\vee) = 0$ the result follows.
\end{proof}
This result can be improved, see for instance Proposition 3.3 and Lemma 3.9
in \autocite{PR1987}.  We are now in a position to show that the
tautological bundle is stable.
\begin{proposition}
  $S_C^\vee$ is stable.  
\end{proposition}
\begin{proof}
  Recall that $S_C^\vee$ is globally generated of slope $\mu(S_C^\vee) = 7/2$.
  Let
  \begin{equation}
    \label{eq:destabilizing-sequence}
    0 \rightarrow F \rightarrow S_C^\vee \rightarrow M \rightarrow 0
  \end{equation}
  be an exact sequence of vector bundles.
  We distinguish several possibilities, depending on the rank of $M$.
  In each case we have to prove that $\mu(M) > \mu(S_C^\vee) = \frac{7}{2}$.
  \begin{itemize}
  \item $\rk(M) = 1$:
    Since $S_C^\vee$ is globally generated, so is $M$.  Hence $h^0(C, M) \geq 2$
    which implies $\mu(M) = \deg(M) \geq 5$, for $W^1_4(C) = \emptyset$.
  \item $\rk(M) = 2$:  Again $M$ is generated by global sections
    and we have $H^0(C, M^\vee) = 0$ by dualizing the exact sequence \eqref{eq:destabilizing-sequence}.
    Using Lemma \ref{lem:determinant-sections} we have
    $h^0(C, \det M) \geq 3$ which implies $\deg(M) \geq 8$ (there is no $\mathfrak{g}^2_7$ on $C$).
    Hence $\mu(M) \geq 4$.
  \item $\rk(M) = 3$:  Now $F$ is a line bundle, more concretely $F = \omega_C \otimes L^{-1}$
    where $L = \det(M)$.  Again by Lemma \ref{lem:determinant-sections}
    we have $h^0(C, L) \geq 4$.  If the inequality is strict, then $\deg(M) \geq 11$
    since $C$ has no $\mathfrak{g}^4_{10}$.  Hence $\mu(M) \geq \frac{11}{3} > \frac{7}{2}$
    and we are done.

    So assume $h^0(C, L) = 4$.  Then $M$ is actually the dual $Q_L$ of the kernel bundle of $L$.
    Since $h^0(C, F) + h^0(C, M) \geq h^0(C, S_C^\vee) = 6$ we have $h^0(C, F) \geq 2$
    and therefore $\deg(F) \geq 5$, i.e., $\deg(M) \leq 9$.
    Since $C$ has no $\mathfrak{g}^3_8$ we must have $\deg(L) \geq 9$.
    Together this implies that $\deg(M) = 9$, i.e., $L \in W^3_9(C)$ and $F \in W^1_5(C)$.
    This means $S_C^\vee$ sits in the exact sequence
    \begin{equation*}
      0 \rightarrow F \rightarrow S_C^\vee \rightarrow Q_L \rightarrow 0.
    \end{equation*}
    But from Lemma \ref{lem:morph-to-g15} we also have
    \begin{equation*}
      0 \rightarrow Q_L \rightarrow S_C^\vee \rightarrow F \rightarrow 0.
    \end{equation*}
    So we have the composition
    $\varphi\colon Q_L \xrightarrow{\beta} S_C^\vee \xrightarrow{\alpha} Q_L$
    which is nonzero since $\ker(\alpha) = F$.  Because $Q_L$
    is stable, $\varphi$ is a homothety and in particular invertible.
    Then $\beta\circ\varphi^{-1}$ induces a splitting
    $S_C^\vee = F \oplus Q_L$.
    But since we have
    \begin{equation*}
      0 \rightarrow S_C \rightarrow H^0(C, E_C) \otimes \mathcal{O}_C \rightarrow E_C \rightarrow 0
    \end{equation*}
    and $E_C$ is nonsplit, this is a contradiction.\qedhere
  \end{itemize}
\end{proof}

\subsection{Stability of the normal bundle}
\label{subsec:stability-normal-bundle-g8}

We have established that $S_C^\vee$ is stable.
In order to prove the same for $\ExplicitNormalBundle$,
we need some heavy machinery.
\begin{theorem}[\autocite{Kobayashi1982,Luebke1983,Donaldson1985,UhlenbeckYau1986}]
  \label{thm:kluduy}
  Let $E$ be a vector bundle on a complex smooth projective variety.
  If $E$ is Hermite--Einstein then it is polystable.
  If $E$ is stable then it is an Hermite--Einstein bundle.
\end{theorem}
\begin{corollary}
  If $E$ is a stable vector bundle on a complex smooth projective variety,
  then $\wedge^q E$ is polystable.  Analogous statements hold
  for $\Sym^q E$ and tensor products of stable vector bundles.
\end{corollary}
\begin{proof}
  $E$ is Hermite--Einstein by Theorem \ref{thm:kluduy},
  hence $\wedge^q E$ is as well (see e.g.\@ \autocite{Luebke1983}).
  The result then follows by using Theorem \ref{thm:kluduy} again.
\end{proof}
Note that since we are on a curve, we could just as well use
the Narasimhan--Seshadri theorem \autocite{NS1965}, which is twenty years earlier
and less involved.  With the eventual study of normal bundles
of polarized K3 surfaces in mind,
however, we have cited here the more general result.

Hence we know $\MyNormalBundle(-1) = \wedgesc$ is polystable
and we are left with proving that it is indecomposable.
We do this by excluding all possible ranks of bundles
that could appear in a splitting.

Observe first that the sequence
\begin{equation*}
  0 \rightarrow Q_\eta \rightarrow S_C^\vee \rightarrow \zeta \rightarrow 0
\end{equation*}
from Lemma \ref{lem:morph-to-g15}
for $\zeta \in W^1_5(C)$ and $\eta$ the Serre dual, leads to
\begin{equation}
  \label{eq:fundamental-exact-sequence}
  0 \rightarrow Q_\eta^\vee \otimes \eta \rightarrow \wedgesc
  \rightarrow Q_\eta \otimes \zeta \rightarrow 0.
\end{equation}
This exact sequence will be fundamental in the proof.
We will also need the following two facts:
\begin{lemma}
  \label{lem:normalbundle-no-g15}
  Let $\xi \in W^1_5(C)$.  Then $H^0(C, \MyNormalBundle(-1) \otimes \xi^{-1}) = 0$.
\end{lemma}
\begin{proof}
  Recall that $\MyNormalBundle(-1) = \MyNormalBundle^\vee(2)$.  We have the standard
  identification
  \[H^0(C, \MyNormalBundle^\vee(2)) = I_2(K_C)\]
  i.e., the global
  sections of $\MyNormalBundle^\vee(2)$ correspond to the quadrics in $\PP^7$ containing the canonical curve.
  If $\mu = K_C \otimes \xi^{-1}$ then $H^0(C, \MyNormalBundle(-1) \otimes \xi^{-1}) \not= 0$
  implies that the multiplication map
  \begin{equation*}
    \Sym^2 H^0(C, \mu) \rightarrow H^0(C, \mu^{\otimes 2})
  \end{equation*}
  is not injective.  But $\mu \in W^3_9(C)$ and if the image under the induced embedding
  were contained in a quadric surface $X$, then a ruling of $X$ would induce a
  $\mathfrak{g}^1_4$ on $C$.  By assumption $W^1_4(C) = \emptyset$, so this is a contradiction.
\end{proof}
\begin{lemma}
  \label{lem:sublinebundle-deg1}
  Every sub-line bundle $L$ of $Q_\eta$ has degree at most $1$.
  Every subbundle of $Q_\eta$ of rank $2$ has degree at most $4$.
\end{lemma}
\begin{proof}
  Consider an exact sequence
  \begin{equation*}
    0 \rightarrow L \rightarrow Q_\eta \rightarrow G \rightarrow 0
  \end{equation*}
  with $L$ a line bundle.  By increasing the degree of $L$, if needed,
  we may assume that $G$ is a vector bundle as well.
  Then by Lemma \ref{lem:determinant-sections} we have $h^0(C, \det(G)) \geq 3$
  and since $W^2_7(C) = \emptyset$ we must have $\deg(G) \geq 8$.
  Hence $\deg(L) \leq 1$.

  If $0 \rightarrow G \rightarrow Q_\eta$ is a rank $2$ subbundle of
  maximal degree then the quotient is a globally generated line bundle $L$,
  hence $\deg(L) \geq 5$.  This implies $\deg(G) \leq 4$.
\end{proof}
\begin{theorem}
  The normal bundle of a genus $8$ curve with $W^2_7(C) = \emptyset$ is stable.
\end{theorem}
\begin{proof}
  We have to show $\wedgesc$ does not decompose into a direct sum of stable bundles
  of slope $\mu(\wedgesc) = 7$.
  Assume that we can write
  $\wedgesc = F\oplus G$ where $F$ is stable of slope $\mu(F) = 7$.  We fix some notation.
  Let $i_F\colon F \hookrightarrow F\oplus G$
  be the inclusion and denote by $q$ the surjective map $\wedgesc \rightarrow Q_\eta \otimes \zeta$
  from equation \eqref{eq:fundamental-exact-sequence}.  Let
  \begin{equation*}
    \varphi = q \circ i_F \colon F \rightarrow Q_\eta \otimes \zeta.
  \end{equation*}
  We claim that $\rk(\varphi)$ cannot be $1$.

  Assume the contrary.  Then $B = \image(\varphi)$ is a line bundle and
  we have exact sequences
  \begin{equation*}
    0 \rightarrow \ker(\varphi) \rightarrow F \rightarrow B \rightarrow 0
  \end{equation*}
  and
  \begin{equation*}
    0 \rightarrow B \rightarrow Q_\eta \otimes \zeta \rightarrow \coker(\varphi) \rightarrow 0.
  \end{equation*}
  Lemma \ref{lem:sublinebundle-deg1} implies that every sub-line bundle of $Q_\eta \otimes \zeta$
  has degree at most $6$, hence $\deg(B) \leq 6$ as well.
  But $F$ is stable and has slope $\mu(F) = 7$, hence $\deg(B) \geq 7$
  (we even have $\deg(B) \geq 8$ if $\rk(F) \geq 2$).  This is impossible.

  Now we deal with the case $\rk(\varphi) = 0$, i.e., $\varphi = 0$.
  This means that $F$ is contained in the kernel of $q$, which is $Q_\eta^\vee \otimes \eta$, that is
  we get $0 \rightarrow F \rightarrow Q_\eta^\vee \otimes \eta$.
  This is impossible since $\mu(F) = 7$ and $\mu(Q_\eta^\vee \otimes \eta) = 6$,
  but $Q_\eta^\vee \otimes \eta$ is stable.

  We now analyze all possible cases for the rank of $F$.
  By swapping $F$ and $G$, if necessary, we may assume $F$ is stable of rank $1$, $2$ or $3$.
  The case of $\rk(F) = 1$ can already be excluded, since then $\rk(\varphi) \leq 1$.

  If $\rk(F) = 2$ then we must have $\rk(\varphi) = 2$, i.e.,
  \begin{equation*}
    0 \rightarrow F \rightarrow Q_\eta \otimes \zeta \rightarrow B \rightarrow 0
  \end{equation*}
  where $B$ is a line bundle, since every rank $2$ subbundle of $Q_\eta \otimes \zeta$
  has slope at most $7$.  Tensoring by $\zeta^{-1}$ we obtain
  \begin{equation*}
    0 \rightarrow F\otimes \zeta^{-1} \rightarrow Q_\eta
    \rightarrow B \otimes \zeta^{-1} \rightarrow 0
  \end{equation*}
  with $\deg(B \otimes \zeta^{-1}) = 5$, hence $h^0(C, B \otimes \zeta^{-1}) = 2$.
  Comparing this to $h^0(C, Q_\eta)$ we find $h^0(C, F \otimes \zeta^{-1}) \geq 2$
  and hence $h^0(C, \MyNormalBundle(-1) \otimes \zeta^{-1}) \geq 2$ in contradiction of
  Lemma \ref{lem:normalbundle-no-g15}.

  Finally, consider the case $\rk(F) = 3$.
  If $\rk(\varphi) = 2$ then we have the sequences
  \begin{equation*}
    0 \rightarrow \ker(\varphi) \rightarrow F \rightarrow \image(\varphi) \rightarrow 0, \quad
    0 \rightarrow \image(\varphi) \rightarrow Q_\eta \otimes \zeta
    \rightarrow \coker(\varphi) \rightarrow 0
  \end{equation*}
  which imply $\mu(\image(\varphi)) > 7$ and $\mu(\image(\varphi)) \leq 7$, respectively
  ($Q_\eta$ does not contain a rank $2$ subbundle of degree at least $5$).
  The last remaining possibility is then $\rk(\varphi) = 3$, i.e.,
  $0 \rightarrow F \rightarrow Q_\eta \otimes \zeta$.  The quotient is a torsion sheaf of length $3$.
  Tensoring the sequence by $\zeta^{-1}$ and comparing global sections,
  we again find $h^0(C, F \otimes \zeta^{-1}) \geq 1$ in contradiction to
  $h^0(C, \MyNormalBundle(-1) \otimes \zeta^{-1}) = 0$.
\end{proof}

\section{Polystability of the normal bundle on curves with a $\mathfrak{g}^2_7$}
\label{sec:semistability-g27}

Ide and Mukai describe canonical models of genus $8$ curves having a $\mathfrak{g}^2_7$ in \autocite{MI2003}.
As long as $W^1_4(C) = \emptyset$, the scheme $W^2_7(C)$ has length exactly $2$.
In the general case, there are precisely two non-autoresidual $\mathfrak{g}^2_7$,
say $W^2_7(C) = \{ \alpha, \beta \}$ with $\alpha \not= \beta$.
We then get a map $\varphi\colon C \rightarrow \PP^2 \times \PP^2$ induced by the linear
system $|\alpha| \times |\beta|$.  It is shown in \autocite{MI2003} that $\varphi$ is an embedding and
its image a complete intersection of divisors of type $(1,2)$, $(2,1)$ and $(1,1)$.

Let $\mathcal{S}$ be the image of the Segre embedding
$\sigma\colon \PP^2 \times \PP^2 \rightarrow \PP^8$.
Then the canonical model of $C$ lies in $\sigma(\PP^2 \times \PP^2)$ intersected with
a hyperplane $H = \PP^7$.  We let $W = \mathcal{S} \cap H$.
So inside $W$, $C$ is the transversal intersection of a divisor $D_{1,2}$ of type $(1,2)$
and one divisor $D_{2,1}$ of type $(2,1)$.
We have the exact sequence
\begin{equation*}
  0 \rightarrow \NormalBundle_{C/W} \rightarrow \NormalBundle_{C/\PP^7}
  \rightarrow \NormalBundle_{W/\PP^7}|_C \rightarrow 0.
\end{equation*}
Since $C$ is a complete intersection in $W$, the normal bundle is just
\begin{equation*}
  \NormalBundle_{C/W} =
  \mathcal{O}_{\PP^2\times\PP^2}(1,2)|_C \oplus \mathcal{O}_{\PP^2\times\PP^2}(2,1)|_C
  = \omega_C \otimes (\alpha \oplus \beta)
\end{equation*}
or equivalently $\NormalBundle_{C/W}(-1) = \alpha \oplus \beta$ under the canonical embedding
with $\mathcal{O}_C(1) = \omega_C$.
We immediately get that $\NormalBundle_{C/\PP^7}$ is not stable,
since it contains $\alpha\oplus \beta$ as a subbundle.

The following fact seems to be well-known and follows from the more general
description of tangent bundles of flag varieties:
\begin{lemma}
  The normal bundle of $\mathcal{S}$ in $\PP^8$ is
  \begin{equation*}
    K_{(1,0)} \otimes K_{(0,1)} \otimes \mathcal{O}_{\mathcal{S}}(1,1)
  \end{equation*}
  where $K_{(i,j)}$ is the kernel of the evaluation map
  $H^0(\mathcal{S}, \mathcal{O}_{\mathcal{S}}(i,j)) \otimes \mathcal{O}_{\mathcal{S}}
  \rightarrow \mathcal{O}_{\mathcal{S}}(i,j)$.
\end{lemma}
Using $\NormalBundle_{W/H} = \NormalBundle_{\mathcal{S}/\PP^8}|_W$ and
pulling this back to $C$ we obtain that
$\NormalBundle_{W/\P^7}|_C(-1)$ is equal to $Q_\alpha \otimes Q_\beta$.
Here $Q_\alpha$ and $Q_\beta$ are the duals of the kernel bundles associated
to $\alpha$ and $\beta$, respectively.
Hence $\MyNormalBundle(-1)$ sits in the exact sequence
\begin{equation*}
  0 \rightarrow \alpha \oplus \beta \rightarrow \MyNormalBundle(-1)
  \rightarrow Q_\alpha \otimes Q_\beta \rightarrow 0
\end{equation*}
and because $\MyNormalBundle$ is self-dual up to twist, this sequence splits.  Therefore
\begin{equation}
  \label{eq:normalbundle-g27-split}
  \MyNormalBundle(-1) = \alpha \oplus ( Q_\alpha \otimes Q_\beta ) \oplus \beta.
\end{equation}
To prove that this is in fact polystable, we will show the stability of $Q_\alpha \otimes Q_\beta$.
Some preliminary results will be needed.
\begin{lemma}
  $Q_\beta$ is stable and $h^0(C, Q_\beta) = 3$.  The maximal degree of a line subbundle
  of $Q_\beta$ is $2$.  Furthermore,
  \begin{equation*}
    H^0(C, Q_\beta \otimes \beta^{-1}) = H^0(C, Q_\beta \otimes \alpha^{-1}) = 0.
  \end{equation*}
  Analogous statements hold for $Q_\alpha$.
\end{lemma}
\begin{proof}
  Stability is shown as in Lemma \ref{lem:stability-kernel-g39}.
  Now $Q_\beta$ sits in the exact sequence
  \begin{equation*}
    0 \rightarrow \mathcal{O}_C(a+b) \rightarrow Q_\beta
    \rightarrow \beta \otimes \mathcal{O}_C(-a-b) \rightarrow 0
  \end{equation*}
  where $a,b \in C$ are map to a singular point of the degree $7$ plane model induced by $\beta$,
  hence $\beta\otimes \mathcal{O}_C(-a-b)$ is a $\mathfrak{g}^1_5$.
  This shows $h^0(C, Q_\beta) = 3$ and $\deg(L) \leq 2$ for every line subbundle
  $L$ of $Q_\beta$.  To see the last claim, tensor with $\beta^{-1}$ or $\alpha^{-1}$
  and take global sections.
\end{proof}
\begin{remark}
  Since $h^0(C, \MyNormalBundle) = 15$, equation \eqref{eq:normalbundle-g27-split}
  yields $h^0(C, Q_\alpha \otimes Q_\beta) = 9$ as a consequence.
  One can also proceed similarly to \autocite{Mukai2010}, section 5,
  to prove this result directly.
\end{remark}
\begin{lemma}
  $Q_\alpha \otimes Q_\beta$ is stable.
\end{lemma}
\begin{proof}
  As a tensor product of stable bundles we already know $Q_\alpha \otimes Q_\beta$
  is polystable (see Theorem \ref{thm:kluduy}) of slope $\mu(Q_\alpha \otimes Q_\beta) = 7$.

  Consider the exact sequence
  \begin{equation*}
    0 \rightarrow Q_\alpha(a+b) \xrightarrow{p} Q_\alpha \otimes Q_\beta
    \xrightarrow{q} Q_\alpha \otimes \xi \rightarrow 0
  \end{equation*}
  where $\beta(-a-b) = \xi \in W^1_5(C)$ and assume
  there exists a line bundle direct summand $L$ of $Q_\alpha \otimes Q_\beta$.
  The induced map $L \rightarrow Q_\alpha\otimes\xi$ cannot be zero,
  since otherwise we get $L \hookrightarrow \ker(q) = Q_\alpha (a+b)$,
  contradicting the stability of $Q_\alpha(a+b)$.
  So $L \rightarrow Q_\alpha \otimes \xi$ is injective, hence we obtain
  \begin{equation*}
    0 \rightarrow L \rightarrow Q_\alpha \otimes \xi \rightarrow B \rightarrow 0
  \end{equation*}
  where $B$ is a line bundle of degree $10$.  Looking at global sections,
  $h^0(C, L \otimes \xi^{-1}) = 1$ and hence $\xi$ is a subbundle of $L$.
  But $L$, as a direct summand of the globally generated bundle $Q_\alpha \otimes Q_\beta$,
  is itself globally generated, so $L \in W^2_7(C)$.
  If $L = \alpha$ then
  \begin{equation*}
    H^0(C, Q_\alpha \otimes Q_\beta \otimes L^{-1})
    = H^0(C, Q_\alpha^\vee \otimes Q_\beta) = \Hom(Q_\alpha, Q_\beta) = 0
  \end{equation*}
  by the stability of $Q_\alpha$ and $Q_\beta$ and the fact that $Q_\alpha$ and $Q_\beta$
  are not isomorphic.  The same statement holds if $L = \beta$.
  So $Q_\alpha \otimes Q_\beta$ does not contain a line bundle as direct summand.

  Assume now that $Q_\alpha \otimes Q_\beta = F \oplus G$ where $F$ and $G$
  are two rank $2$ bundles.
  Consider again the sequence
  \begin{equation*}
    0 \rightarrow Q_\alpha(a+b) \xrightarrow{p} F \oplus G
    \xrightarrow{q} Q_\alpha \otimes \xi \rightarrow 0
  \end{equation*}
  and the induced maps $\varphi\colon Q_\alpha(a+b) \rightarrow F$
  and $\psi\colon Q_\alpha(a+b) \rightarrow G$.  If $\varphi = 0$
  this would induce a surjection $Q_\alpha \otimes \xi \rightarrow F \rightarrow 0$,
  which is impossible.
  If the rank of $\varphi$ was $1$ then we would get
  \begin{equation*}
    Q_\alpha (a+b) \rightarrow \image(\varphi) \rightarrow 0, \quad
    0 \rightarrow \image(\varphi) \rightarrow F
  \end{equation*}
  and hence $\deg(\image(\varphi)) \geq 7$ and $\deg(\image(\varphi)) < 7$, respectively.
  So the only possibility is that $\varphi$ is injective.  The same reasoning
  applies to $\psi$, hence $Q_\alpha(a+b) \hookrightarrow F$
  and $Q_\alpha(a+b) \hookrightarrow G$.  This implies $h^0(C, F(-a)) \geq 3$, $h^0(C, G(-a)) \geq 3$,
  i.e., $h^0(C, Q_\alpha \otimes Q_\beta \otimes \mathcal{O}_C(-a)) \geq 6$.
  But $Q_\alpha \otimes Q_\beta$ is globally generated of rank $4$, so
  \begin{equation*}
    h^0(C, Q_\alpha \otimes Q_\beta \otimes \mathcal{O}_C(-a)) = h^0(C, Q_\alpha \otimes Q_\beta) - 4 = 5
  \end{equation*}
  which is again a contradiction.
\end{proof}
\begin{remark}
  Since we have $E_C = \alpha \oplus \beta$ for the Mukai bundle on curves with a
  $\mathfrak{g}^2_7$ but no $\mathfrak{g}^1_4$ we can still form the exact sequence
  \begin{equation*}
    0 \rightarrow S_C \rightarrow H^0(C, E_C) \otimes \mathcal{O}_C \rightarrow E_C \rightarrow 0,
  \end{equation*}
  even though it does not come from a tautological exact sequence over some Grassmannian,
  and then $S_C^\vee = Q_\alpha \oplus Q_\beta$.  We also get
  \begin{equation*}
    \wedgesc = \alpha \oplus (Q_\alpha \otimes Q_\beta) \oplus \beta
  \end{equation*}
  which is exactly the normal bundle $\MyNormalBundle(-1)$ of the canonical embedding of $C$.
  So the normal bundle has intrinsically the same description (in terms of the Mukai bundle)
  as on the general curve.
\end{remark}



\appendix
\section{The normal bundle in genus $6$}
\label{sec:app-genus6}

Every genus $6$ curve is tetragonal, hence by \autocite{AFO2016}, Proposition 3.2,
we know that the normal bundles of canonical genus $6$ curves are never semistable.
The proof in loc.\@ cit.\@ uses a rational scroll induced by the tetragonal pencil.
In the spirit of the previous sections, we are going to give an alternative proof
for the general curve using Mukai's description of genus $6$ curves as quadric
sections of Grassmannians.

By \autocite{Mukai1993}, section 5, a canonical genus $6$ curve $C$ is an
intersection of $G = \Grassmannian{2}{5} \subseteq \PP^9$ with a $4$-dimensional quadric
if and only if $W^1_4(C)$ is finite, i.e., if $C$ is not trigonal, bielliptic or a plane quintic.
The intersection of $G$ with a $\PP^5 = H \subseteq \PP^9$ is a del Pezzo surface $X$ and 
then $C$ is a quadric hypersurface section of $S$.

The inclusion $C \hookrightarrow G$ in the Grassmannian naturally
leads to a destabilizing sequence for the normal bundle $\mathcal{N}_{C/\PP^5}$.
Let $S$ and $Q$ be the tautological bundle and the universal quotient bundle of $G$, respectively.
We then have (as in section \ref{subsec:normalbundle-g8-description})
\begin{equation*}
  \mathcal{N}_{X/\PP^5} \cong \mathcal{N}_{\Grassmannian{2}{5}/\PP^9}|_X
  \cong (\wedge^2 S^\vee) \otimes \det Q.
\end{equation*}
Furthermore, the exact sequence of normal bundles
\begin{equation*}
  0 \rightarrow \mathcal{N}_{C/X} \rightarrow \mathcal{N}_{C/\PP^5}
  \rightarrow \mathcal{N}_{X/\PP^5}|_C \rightarrow 0
\end{equation*}
is split, i.e.,
\begin{equation*}
  \mathcal{N}_{C/\PP^5} \cong \mathcal{N}_{C/X} \oplus \mathcal{N}_{X/\PP^5}|_C
                        = \Big( \wedgesc \Big) \otimes \omega_C.
\end{equation*}
Here $S_C$ denotes the restriction of $S$ to $C$.  Since $S_C^\vee$ is of rank $3$ we have
\begin{equation*}
  \wedgesc \cong S_C \otimes \det( S_C^\vee ) \cong S_C \otimes \omega_C
\end{equation*}
implying $\mathcal{N}_{X/\PP^5}|_C = S_C \otimes \omega_C^{\otimes 2}$.
Now this is a vector bundle of slope
$(\deg \omega_C^{\otimes 5})/3 = 50/3$
and it is a direct summand of $\mathcal{N}_{C/\PP^5}$, which has slope
$(\deg \omega_C^{\otimes 7})/4 = 70/4$.  Since $50/3 < 70/4$, the normal bundle is unstable.

\section{The genus $9$ case}

The next case to consider in order to gain information
on Conjecture \ref{conj:afo} would be canonical curves of genus $9$.
For the general (i.e.\@ non-pentagonal) curve we again have a
Grassmannian embedding, this time
into the symplectic Grassmannian $\SpG(3,6)$ 
\begin{equation*}
  C \hookrightarrow \SpG(3,6) \subseteq \PP^{13}.
\end{equation*}
The embedding is again given by a uniquely defined Mukai bundle $E_C$,
here of rank $3$.  The details can be found in \autocite{Mukai2010}.
$C$ can be recovered as a transerval intersection of $\SpG(3,6)$
with an $8$-plane.

We may again use this description to obtain more information about
the normal bundle of the canonical embedding $C \hookrightarrow \PP^8$.
However, the situation seems to be considerably more complicated
and our methods run out of steam before we can fully prove stability.

Let us consider what information about $\NormalBundle_{C/\PP^8}$
we can get from the Grassmannian embedding.
$\SpG(3,6)$ is a subvariety of the
Grassmannian $G = \Grassmannian{3}{6}$ of $3$-planes in $H^0(C, E_C)$.
On $G$ we have the universal exact sequence
\begin{equation*}
  0 \rightarrow S_G \rightarrow V \otimes \mathcal{O}_G \rightarrow Q_G \rightarrow 0
\end{equation*}
where $Q_G$ is the universal quotient bundle and $S_G$ is the tautological bundle.
The sequence restricts to $X = \SpG(3,6)$ and on $X$ we can identify $S_X$ with $Q_X^\vee$:
\begin{equation*}
  0 \rightarrow Q_X^\vee \rightarrow V \otimes \mathcal{O}_X \rightarrow Q_X \rightarrow 0.
\end{equation*}
$X$ is embedded in $\PP^{13}$ by a restricted Plücker embedding.
\begin{lemma}[{\autocite[Section 2]{Mukai2010}}]
  The normal bundle of $\SpG(3,6) \subseteq \PP^{13}$ sits in the exact sequence
  \begin{equation*}
    0 \rightarrow \mathcal{O}_X \rightarrow \NormalBundle_{X/\PP^{13}}^\vee(2)
    \rightarrow \Sym^2 Q_X \rightarrow 0
  \end{equation*}
  where $\mathcal{O}_X(1) = \det Q_X$.
\end{lemma}
Since $C$ is the intersection $C = \PP^8 \cap X$, we get that the
normal bundle $\NormalBundle_{C/\PP^8}$ of the canonical embedding
is equal to the normal bundle $\NormalBundle_{X/\PP^{13}}$ of $X$,
restricted to $C$.
Furthermore, the restriction of $Q_X$ to $C$ is the Mukai bundle $E_C$.
Thus we obtain the following exact sequence for the normal bundle of $C$.
\begin{lemma}
  The twisted conormal bundle of $C \subseteq \PP^8$ in its canonical embedding
  sits in the exact sequence
  \begin{equation}
    \label{eq:normal-bundle-g9-sequence}
    0 \rightarrow \mathcal{O}_C \rightarrow \NormalBundle_{C/\PP^8}^\vee(2)
    \rightarrow \Sym^2 E_C \rightarrow 0.
  \end{equation}
\end{lemma}
Observe that the universal exact sequence of $X$ restricts to
\begin{equation*}
  0 \rightarrow E_C^\vee \rightarrow H^0(C, E_C) \rightarrow E_C \rightarrow 0
\end{equation*}
on $C$.

Already it is a non-trivial calculation to show that
\begin{equation*}
  \Sym^2 H^0(C, E_C) \rightarrow H^0(C, \Sym^2 E_C)
\end{equation*}
is an isomorphism and hence $h^0(C, \Sym^2 E_C) = 21$.
Having established this, we can make use of representation theory
(see \autocite[Exercise 15.36]{FH1991}) to identify the global sections
of $\Sym^2 E_C$ with the space $I_2(K_C)$ of quadrics containing
the canonical embedding of $C$.

One can then use the exact sequence \eqref{eq:normal-bundle-g9-sequence}
to prove that $\NormalBundle_{C/\PP^8}$
has no destabilizing quotients of rank $1$ or $2$.
Furthermore, one can obtain bounds for the degrees of quotient bundles of other ranks.
The calculations are similar to the genus $8$ case, just more lengthy.
However, we fear that these methods may not be potent enough to establish
stability with respect to quotient bundles of all ranks.

As a final note we present here some easy facts that
might prove useful in tackling the question of stability of the
normal bundle of a general canonical curve for arbitrary $g$.
First we establish normal generation for twists of the normal
and the conormal bundle.  We are thankful to Pete Vermeire for
pointing out simple arguments that hold in greater generality.
\begin{lemma}
  $\NormalBundle_{C/\PP^{g-1}}^\vee(2)$ is globally generated
  for every canonical curve with Clifford index $\Cliff(C) \geq 2$.
\end{lemma}
\begin{proof}
  First note that $\Cliff(C) \geq 2$ implies that $C$
  is scheme-theoretically cut out by quadrics.
  Now, in fact, $\NormalBundle_{X/\PP}^\vee(2)$
  is globally generated for every smooth projective variety $X$
  embedded in a projective space $\PP$, such that $X$ is scheme-theoretically
  cut out by quadrics.

  Indeed, the assumption implies that 
  the ideal sheaf $\mathcal{I}_X(2)$ is globally generated
  and then of course the quotient $\mathcal{I}_X/\mathcal{I}_X^2 (2)$
  is globally generated as well.
\end{proof}
\begin{lemma}
  $\NormalBundle_{C/\PP^{g-1}}(-1)$ is globally generated for all canonical
  curves of every genus $g$.
\end{lemma}
\begin{proof}
  In fact this is true for every smooth projective varieties
  $X$ embedded in projective space by some line bundle $L$.
  First recall that the dual of the evaluation sequence
  \begin{equation*}
    0 \rightarrow M_L \rightarrow H^0(X, L) \otimes \mathcal{O}_X \rightarrow L \rightarrow 0
  \end{equation*}
  shows that $Q_L = M_L^\vee$ is globally generated.
  Furthermore, the pullback to $X$ of the Euler sequence on $\PP \coloneqq \PP H^0(X, L)^\ast$
  identifies $\Omega_{\PP}|_X$ and $M_L(-1)$.

  Now consider the conormal exact sequence
  \begin{equation*}
    0 \rightarrow \NormalBundle_{X/\PP}^\vee
    \rightarrow \Omega_{\PP}|_X \rightarrow \Omega_X \rightarrow 0.
  \end{equation*}
  After dualizing and twisting by $\mathcal{O}_X(-1)$, this can be rewritten as
  \begin{equation*}
    0 \rightarrow \mathcal{T}_X(-1) \rightarrow
    Q_L \rightarrow \NormalBundle_{X/\PP}(-1) \rightarrow 0,
  \end{equation*}
  from which global generation of $\NormalBundle_{X/\PP}(-1)$ follows.
\end{proof}
This has an interesting consequence.
\begin{lemma}
  For every $r_0 > 0$ there is an integer $g_0$ such that the normal bundle
  $\NormalBundle_{C/\PP^{g-1}}$ of a general canonical curve of every genus
  $g \geq g_0$ has no destabilizing quotient bundle of rank $r \leq r_0$.
\end{lemma}
\begin{proof}
  Let $\NormalBundle_{C/\PP^{g-1}}(-1) \rightarrow F \rightarrow 0$ be
  a destabilizing quotient of rank $r$.
  Then $F$ is globally generated and has slope
  \begin{equation*}
    \mu(F) \leq \mu\big( \NormalBundle_{C/\PP^{g-1}}(-1) \big) = 6 \left( \frac{g - 1}{g - 2} \right)
  \end{equation*}
  We may assume $g \geq 8$.  Then in particular, $d = \deg(F) \leq 7r$.
  By Lemma \ref{lem:determinant-sections} we must have $h^0(C, \det(F)) \geq r + 1$.
  But if $g$ is large enough, then $\rho(g,r,d) < 0$.  
  We conclude that for these $g$ no such $F$ exists.
\end{proof}
Finally, let us note that in general it seems hard to decide
when precisely the normal bundle of an arbitrary (non-general) curve $C$ is unstable.
Not having a tetragonal pencil does not seem to be the right condition:
\begin{example}
  The normal bundle of a canonical curve can be unstable even if the
  curve is not tetragonal.
  Consider a curve of genus $9$ with a $\mathfrak{g}^3_9 \eqqcolon L$ but no $\mathfrak{g}^1_4$.
  Such curves are explicitly constructed in M.\@ Sagraloff's Ph.D.\@ thesis
  \autocite{Sagraloff2006}.
  We have $\rho(9,3,9) = -3$ as well as $\rho(9,1,4) = -3$, but the two
  Brill--Noether loci $\MgBNLocus{9}{4}{1}$ and $\MgBNLocus{9}{9}{3}$ are not equal.

  Consider the Serre dual $A = \omega_C \otimes L^{-1}$ of $L$, which is a $\mathfrak{g}^2_7$.
  By using $L$ and $A$, we obtain an embedding of $C$ into $\PP^2 \times \PP^3$.
  By composing this with the Segre embedding, we recover the canonical image of $C$.
  Hence we get a quotient $Q_L \otimes Q_A$ of $\NormalBundle_{C/\PP^8}(-1)$
  of slope $\mu(Q_A\otimes Q_L) = 13/2$.
  However, note that the normal bundle has slope $\mu(\NormalBundle_{C/\PP^8}(-1)) = 48/7$,
  which is strictly bigger.
\end{example}

\printbibliography

\Addresses

\end{document}